\documentclass{amsart}
\usepackage{amsfonts}
\usepackage{amsmath,amssymb}
\usepackage{amsthm}
\usepackage{amscd}
\usepackage{graphics}
\usepackage{graphicx}

\theoremstyle{remark}{
\newtheorem{Def}{{\rm Definition}}
\newtheorem{Ex}{{\rm Example}}
\newtheorem{Rem}{{\rm Remark}}

}
\theoremstyle{plain}{

\newtheorem{Prop}{Proposition}
\newtheorem{Thm}{Theorem}
\newtheorem{MainThm}{Main Theorem}

\newtheorem{Fact}{Fact}
}

\begin{document}
\title[Simple fold maps with good properties on graph manifolds]{Characterizing families of graph manifolds via suitable classes of simple fold maps into the plane and embeddability of the Reeb spaces in some $3$-dimensional manifolds}
\author{Naoki Kitazawa}
\keywords{Fold maps: simple fold maps and round fold maps. Reeb spaces. Polyhedra. Graphs. Multibranched surfaces. Heegaard genera of $3$-dimensional manifolds.\\
\indent {\it \textup{2020} Mathematics Subject Classification}: Primary~57R45. Secondary~57Q05. ~57K20. ~57K30.}
\address{Institute of Mathematics for Industry, Kyushu University, 744 Motooka, Nishi-ku Fukuoka 819-0395, Japan\\
 TEL (Office): +81-92-802-4402 \\
 FAX (Office): +81-92-802-4405 \\
}
\email{n-kitazawa@imi.kyushu-u.ac.jp}
\urladdr{https://naokikitazawa.github.io/NaokiKitazawa.html}
\maketitle
\begin{abstract}
{\it Graph} manifolds form important classes of $3$-dimensional closed and orientable manifolds. For example, {\it Seifert} manifolds are graph manifolds where {\it hyperbolic} manifolds are not.

In applying singularity theory of differentiable maps to understanding global topologies of manifolds, graph manifolds have been shown to be characterized as ones admitting so-called {\it simple} {\it fold} maps into the plane of explicit classes by Saeki and the author. The present paper presents several related new results. 

{\it Fold} maps are higher dimensional variants of Morse functions and {\it simple} ones form simple classes, generalizing the class of general Morse functions. Such maps into the plane on $3$-dimensional closed and orientable manifolds induce quotient maps to so-called {\it simple polyhedra} with no {\it vertices}, which are $2$-dimensional. This is also closely related to the theory of {\it shadows} of $3$-dimensional manifolds. We also discuss invariants for graph manifolds via embeddability of these polyhedra in some $3$-dimensional manifolds.



\end{abstract}


\maketitle
\section{Introduction.}
\label{sec:1}
The class of {\it graph manifolds} is an important class of $3$-dimensional closed and orientable manifolds.
In short, a {\it graph} manifold is a manifold decomposed into total spaces of so-called circle bundles over compact and connected surfaces via tori in the manifold. Most of $3$-dimensional manifolds are so-called {\it hyperbolic} ones and graph manifolds are never hyperbolic. However, the class contains important manifolds such as a $3$-dimensional sphere, so-called {\it Lens spaces}, and so-called {\it Seifert} manifolds.

The present paper is on {\it fold} maps enjoying good properties on graph manifolds into the plane. {\it Fold} maps are higher dimensional variants of Morse functions.

By virtue of \cite{moise} for example, on $1$, $2$ and $3$-dimensional topological manifolds, there exist unique PL structures and differentiable structures. It is also fundamental that for a topological space homeomorphic to a $1$-dimensional or $2$-dimensional, it has a unique PL structure. These are so-called Hauptvermutung. 
The {\it canonical} PL structure for a smooth manifold means a well-known PL structure canonically obtained from the differentiable structure of the manifold and we omit rigorous expositions. We regard these facts as well-known here.

The {\it piecewise smooth category} is known to be equivalent to the PL category and we regard these two categories as essentially same ones: in our paper, "PL" and "piesewise smooth" are same and we will use "PL" mainly.

For a (smooth) manifold and more generally, a polyhedron $X$, $\dim X$ denotes its dimension.
\subsection{Fold maps.}
For a smooth manifold $X$, $T_pX$ denotes the tangent vector space at $p \in X$. For a smooth map $c:X \rightarrow Y$, $dc_p:T_pX \rightarrow T_{c(p)}Y$ denotes the differential at $p$. A {\it singular} point $p \in X$ means a point where the rank of the differential $dc_p:T_pX \rightarrow T_{c(p)}Y$ is smaller than $\min\{\dim X,\dim Y\}$. $S(c)$ denotes the set of all singular points of $c$ and we call this the {\it singular set} of $c$.

\begin{Def}
\label{def:0}
Let $m \geq n \geq 1$ be integers. A {\it fold} map $f$ on an $m$-dimensional closed and smooth manifold into an $n$-dimensional smooth manifold with no boundary is a smooth map such that at each singular point $f$ is of the form $$(x_1,\cdots,x_m) \mapsto (x_1,\cdots,x_{n-1},{\Sigma}_{j=1}^{m-n-i(p)+1} {x_{n-1+j}}^2-{\Sigma}_{j=1}^{i(p)} {x_{m-i(p)+j}}^2)$$ 
for suitable coordinates and an integer $0 \leq i(p) \leq \frac{m-n+1}{2}$.
\end{Def}

\begin{Prop}
For a fold map $f$, $i(p)$ is unique for any $p \in S(f)$ and defined as the {\rm index} of $p$.
Let $F_j(f) \subset S(f)$ denote the set of all singular points whose indices are $j$. $F_j(f)$ is a smooth closed submanifold with no boundary, making $f {\mid}_{F_j(f)}$ a smooth immersion. 
\end{Prop}

\begin{Def}
\label{def:1}
A fold map is said to be {\it simple} if each connected component of each preimage has at most one singular point. A fold map is said to be {\it strongly simple} if the restriction to its singular set is an embedding.  
\end{Def}
In the present paper, theory of Morse functions are fundamental and see \cite{milnor,milnor2} for example.
{\it Morse} functions are simplest fold maps. For general Morse functions at distinct singular points the singular values are distinct. They are also strongly simple fold maps.
For such fundamental properties on fold maps, see \cite{golubitskyguillemin} as an introductory book on the theory of singularities of differentiable maps including Morse functions and fold maps and see also \cite{thom,whitney} as pioneering works on fold maps and so-called {\it generic} smooth maps on smooth manifolds whose dimensions are greater than or equal to $2$ into the plane. \cite{saeki0} presents some pioneering studies on simple fold maps and manifolds admitting them. These studies are followed by studies closely related to our present study such as \cite{saeki2} and a recent work \cite{kitazawasaeki} for example. Related to this, we introduce {\it round} fold maps, introduced in \cite{kitazawa0.1,kitazawa0.2,kitazawa0.3} by the author and discussed further in \cite{kitazawa0.4,kitazawa0.5} by the author. 

${\mathbb{R}}^k$ denotes the $k$-dimensional Euclidean space where ${\mathbb{R}}^1=\mathbb{R}$. 
$\mathbb{N} \subset \mathbb{R}$ denotes the set of all positive integers.
We regard this as a natural smooth manifold and the Riemannian manifold, endowed with the standard Euclidean metric. $||x||$ denotes the distance between $x \in {\mathbb{R}}^k$ and the origin $0 \in {\mathbb{R}}^k$ or equivalently, the value of the standard norm at the vector $x$. $S^k:=\{x \in {\mathbb{R}}^{k+1} \mid ||x||=1.\}$ is the $k$-dimensional unit sphere for $k \geq 0$ and $D^k:=\{x \in {\mathbb{R}}^{k} \mid ||x|| \leq 1.\}$ is the $k$-dimensional unit disk for $k \geq 1$. The $k$-dimensional unit sphere is a $k$-dimensional smooth closed submanifold of ${\mathbb{R}}^{k+1}$ with no boundary and it is connected for $k \geq 1$. The $k$-dimensional unit disk is a smooth, compact and connected submanifold of ${\mathbb{R}}^k$.

\begin{Def}
	\label{def:2}
	Let $n \geq 2$ be an integer. A strongly simple fold map $f$ into ${\mathbb{R}}^n$ is said to be {\it round} if there exists an integer $l>0$ and a diffeomorphism ${\phi}_{{\mathbb{R}}^n}:{\mathbb{R}}^n \rightarrow {\mathbb{R}}^n$ such that ${\phi}_{{\mathbb{R}}^n}(f(S(f)))=\{x \in {\mathbb{R}}^k \mid 1 \leq ||x|| \leq l, ||x||  \in \mathbb{N}\}$.
\end{Def}

A canonical projection of some unit sphere into some Euclidean space mapping $(x_1,x_2) \in S^m \subset {\mathbb{R}}^{m+1}={\mathbb{R}}^n \times {\mathbb{R}}^{m-n+1}$ to $x_1$ is one of simplest round fold maps. This is also a kind of exercises on theory of Morse functions. A {\it round fold map into $\mathbb{R}$} is defined in \cite{kitazawa0.5} as a Morse function obtained by gluing two copies of a Morse function satisfying suitable good properties on the boundaries. We can see that if a manifold admits a round fold map into ${\mathbb{R}}^n$ for some integer $n \geq 2$, then it admits a round fold map into ${\mathbb{R}}^{n^{\prime}}$ for any $1 \leq n^{\prime}<n$. This is also a kind of such exercises.
\subsection{The Reeb spaces of maps and our Main Theorems.}
For a continuous map $c:X \rightarrow Y$, we can define an equivalence relation ${\sim}_c$ on $X$ by the following rule: $x_1 {\sim}_c x_2$ if and only if $x_1$ and $x_2$ are in a same connected component of a same preimage $c^{-1}(y)$ ($y \in Y$). 
\begin{Def}
\label{def:3}
We call the quotient space $W_c:=X/{\sim}_c$ the {\it Reeb space} of $c$.
\end{Def}
Hereafter, $q_c:X \rightarrow W_c$ denotes the canonically obtained quotient map and $\bar{c}$ denotes the map satisfying $c=\bar{c} \circ q_c$, which is well-defined, defined uniquely and continuous.
The following fact is fundamental and presented as a result we can obtain immediately from a main theorem and arguments in \cite{hiratukasaeki, shiota} for example. 

\begin{Fact}
\label{fact:1}
For a fold map $c:X \rightarrow Y$ between smooth manifolds with no boundaries and the canonical PL structures, we can regard $W_c$ as a polyhedron enjoying the following three.
\begin{enumerate}
\item The PL structure of $W_c$ here is induced canonically from $Y$ uniquely.
\item $q_c$ and $\bar{c}$ is regarded as a PL map. For suitable simplicial complexes compatible with the PL structures, these maps are simplicial maps.
\item For each point $p \in W_f-q_f(S(c))$, we can take a small regular neighborhood PL homeomorphic to the $\dim Y$-dimensional unit disk $D^{\dim Y}$.  
\end{enumerate}
\end{Fact}

We also add that for $m>n=2$ in Definition \ref{def:1}, $W_c$ has been shown to be a polyhedron for a so-called {\it stable} fold map in \cite{kobayashisaeki}. This is also shown for {\it stable} maps there.
Note also that any fold map is, by a suitable slight perturbation, deformed to be a {\it stable} fold map. \cite{golubitskyguillemin} also presents the topology of the space of all smooth maps between given two smooth manifolds with no boundaries, which is important in defining the notion of a {\it stable} map, and fundamental and sophisticated theory of stable maps, for example. Last we add that \cite{reeb} is a pioneering paper on Reeb spaces for smooth (Morse) functions.

\begin{Ex}
	For Morse functions on closed manifolds such that at distinct singular points the values are distinct, which are strongly simple, the Reeb space is regarded as a graph.
	Such functions are known to exist densely.
	Strongly simple fold maps are simple and stable. For a canonical projection of the unit sphere $S^m$ into ${\mathbb{R}}^n$ with $m>n$, which is round, the Reeb space is the $n$-dimensional unit disk $D^n$. 

\end{Ex}
A {\it PL} (smooth) bundle means a bundle whose fiber is a polyhedron (resp. smooth manifold) and whose structure group consists of PL homeomorphisms (resp. smooth diffeomorphisms). Hereafter, a {\it diffeomorphism} is smooth or of the class $C^{\infty}$.
\begin{Fact}
\label{fact:2}
A simple fold map $f$ can be also defined as a fold map such that $q_f {\mid}_{S(f)}$ is injective. Furthermore, for each connected component $C$ of $q_f(S(f))$, a regular beighborhood $N(C)$ has the structure of a PL bundle over $C$ whose fiber is PL homeomorphic to $[-1,1]$ or $\{(r \cos \theta,r \sin \theta) \mid 0 \leq r \leq 1, \theta=0,\frac{2}{3}\pi,\frac{4}{3}\pi \} \subset {\mathbb{R}}^2$.
\end{Fact}

\begin{Def}
	\label{def:4}
	In fact \ref{fact:2}, if the PL bundle $N(C)$ is chosen as a trivial one for any $C$, then $f$ is said to be {\it normal}.
\end{Def}

\begin{MainThm}
	\label{mthm:1}
	A graph manifold $M$ always admits a normal simple fold map $f$ into the plane ${\mathbb{R}}^2$ enjoying the following properties.
	\begin{enumerate}
		\item $W_f$ is PL homeomorphic to a $2$-dimensional polyhedron obtained in the following way.
		\begin{enumerate}
			\item Choose finitely many Reeb spaces of {\rm bordered doubled most standard maps induced from $S^m$ into the plane} or {\rm supporting pairs of pants}, defined in Definitions \ref{def:11} and \ref{def:12}.
			\item Do a finite iteration of choosing a pair of {\rm borders} for the previous Reeb spaces which are in mutually distinct Reeb spaces before and gluing them directly via PL homeomorphisms or attaching Reeb spaces of {\rm S-maps supporting annuli} along the {\rm borders}, defined in Definition \ref{def:13}.
		\end{enumerate} 
		\item $\bar{f}$ is represented as the composition of a PL embedding into ${\mathbb{R}}^3$ with the canonical projection to ${\mathbb{R}}^2$.
	\end{enumerate}
\end{MainThm}

In the next section, we introduce existing characterizations of graph manifolds via explicit simple fold maps into ${\mathbb{R}}^2$ by Saeki and the author \cite{kitazawasaeki, saeki2} as Theorem \ref{thm:1}. Main Theorem \ref{mthm:1} and Main Theorem \ref{mthm:2}, which will be presented here, are new characterizations. Main Theorem \ref{mthm:2} needs the following notion. For a round fold map into ${\mathbb{R}}^n={\mathbb{R}}^2$, we have a PL bundle over $S^1$ whose projection is the map from 
 ${({\phi}_{{\mathbb{R}}^2} \circ \bar{f})}^{-1}(\{x \in {\mathbb{R}}^2 \mid ||x|| \geq \frac{1}{2} \})$ onto $S^1$ mapping $x$ to $\frac{1}{||{({\phi}_{{\mathbb{R}}^2} \circ \bar{f})}(x)||} {({\phi}_{{\mathbb{R}}^2} \circ \bar{f})}(x)$ and fiber is a $1$-dimensional polyhedron in the situation of Definition \ref{def:2}. A {\it topologically quias-trivial} round fold map means a round fold map such that the bundle is trivial for a suitable diffeomorphism ${\phi}_{{\mathbb{R}}^n}={\phi}_{{\mathbb{R}}^2}$ in the situation of Definition \ref{def:2}. We present Main Theorem \ref{mthm:2}: a {\it representation graph} for a graph
 manifold is a useful graph in understanding the topology roughly and defined in Definition \ref{def:8}.
 \begin{MainThm}[Theorem \ref{thm:3}]
	\label{mthm:2}
	A graph manifold $M$ admits a topologically quasi-trivial round fold map into ${\mathbb{R}}^2$ such that the Reeb space can be embedded in $S^3$ in the PL category if and only if there exists a representation graph for $M$ being planar.
\end{MainThm}

\subsection{The content of the present paper}
In the next section, we introduce some important properties of simple fold maps and their Reeb spaces. We also introduce ones which are not presented in the present section. We also define {\it graph manifolds} and graphs we can define for graph manifolds ({\it representation graphs}). They are defined immediately from the structures of graph manifolds. This resembles and is different from graphs in \cite{neumann}. This graph has less information of the manifold than the graph of \cite{neumann}. However, this graph is easier to understand the definition and properties. After these presentations, we introduce characterizations of graph manifolds by simple fold maps satisfying additional good conditions in Theorem \ref{thm:1}. The third section is devoted to proofs of Main Theorems. Key tools are construction of explicit fold maps into surfaces on graph manifolds and similar construction on more general $3$-dimensional compact manifolds first used in \cite{saeki2} and later used in \cite{ishikawakoda, kitazawasaeki} for example. \cite{costantinothurston} is also a closely related study.
\section{Some important properties of simple fold maps, graph manifolds and existing characterizations of graph manifolds via simple fold maps satisfying additional good conditions.}

\subsection{Some important properties of simple fold maps.}

The following gives a rigorous exposition of Fact \ref{fact:2}, presented in the introduction or the first section.

\begin{Prop}
\label{prop:2}
For a simple fold map $f$ in {\rm Definition \ref{def:1}} with $m>n \geq 1$, $W_f$ is regarded as a polyhedron in {\rm Fact \ref{fact:1}} and enjoys the following properties.\\\
\begin{enumerate}
\item
\label{prop:2.1}
$W_f-q_f(S(f))$ is also regarded as an $n$-dimensional smooth manifold and the restriction of $\bar{f}$ here is a smooth immersion. 
\item
\label{prop:2.2}
For each connected component of $W_f-q_f(S(f))$, the closure is the disjoint union of the connected component and some connected components of $q_f(S(f))$ and also regarded as a compact and smooth manifold whose interior is a smooth submanifold of the smooth manifold $W_f-q_f(S(f))$ before. The restriction of $\bar{f}$ on the closure before is also a smooth immersion.
\item
\label{prop:2.3}
For each connected component $C_j \subset q_f(S(f))$, we can take some small regular neighborhood $N(C_j)$ satisfying either of the following conditions. 
\begin{enumerate}
\item
\label{prop:2.3.1}
$N(C_j)$ is regarded as the total space of a trivial PL bundle over $C_j$ whose fiber is a closed interval $I_j:=[0,1]$, giving the identification of $C_j \subset q_f(S(f))$ with $C_j \times \{0\} \subset N(C_j)$ via the map mapping $x$ to $(x,0)$. Furthermore, the restriction of $\bar{f}$ to the restriction of the trivial bundle $N(C_j)$ over an arbitrary small closed interval $I_{C_j,j}$ smoothly embedded in $C_j$ is regarded as the product map of two PL homeomorphisms between closed intervals {\rm (}onto the product of the small closed interval smoothly embedded in the manifold $N$ of the target and the image of the singular set and the closed interval $I_j$ for suitable coordinates{\rm )}.
\item
\label{prop:2.3.2}
$N(C_j)$ is regarded as the total space of a trivial PL bundle over $C_j$ whose fiber is a closed interval $I_j:=[-1,1]$, giving the identification of $C_j \subset q_f(S(f))$ with $C_j\times \{0\} \subset N(C_j)$ via the map mapping $x$ to $(x,0)$. Furthermore, the restriction of $\bar{f}$ to the restriction of the trivial bundle $N(C_j)$ over an arbitrary small closed interval $I_{C_j,j}$ smoothly embedded in $C_j$ is regarded as the product map of two PL homeomorphisms between closed intervals {\rm (}onto the product of the small closed interval smoothly embedded in the manifold $N$ of the target and the image of the singular set and the closed interval $I_j$ for suitable coordinates{\rm )}.
\item 
\label{prop:2.3.3}
$N(C_j)$ is regarded as the total space of a PL bundle over $C_j$ whose fiber is $K_j:=\{(r \cos t, r \sin t) \mid 0 \leq r \leq 1, t=0,\frac{2}{3}\pi, \frac{4}{3}\pi.\} \subset {\mathbb{R}}^2$, giving the identification of $C_j$ with $C_j \times \{(0,0)\} \subset N(C_j)$ via the map mapping $x$ to $(x,(0,0))$ and the following two hold.
\begin{enumerate}
\item
\label{prop:2.3.3.1}
 The structure group of the bundle is trivial or of order $2$. In the latter case, the group acts on $K_j$ fixing the set $\{(r,0)\mid 0 \leq r \leq 1.\}$ and having the non-trivial element mapping $(r \cos \frac{2}{3} \pi, r \sin \frac{2}{3} \pi)$ to $(r \cos \frac{4}{3} \pi, r \sin \frac{4}{3} \pi)$ for $0 \leq r \leq 1$. 
\item
\label{prop:2.3.3.2}
 The restriction of $\bar{f}$ to the restriction of the trivial bundle $N(C_j)$ over the arbitrary small closed interval $I_{C_j,j}$ smoothly embedded in $C_j$ is regarded as the surjective product map of the two following maps {\rm (}onto the product of the small closed interval smoothly embedded in the manifold $N$ of the target and the image of the singular set and the $1$-dimensional connected submanifold $\{(r,0) \mid -\frac{1}{2} \leq r \leq 1\}$ for suitable coordinates{\rm )}.
\begin{enumerate}
\item A PL homeomorphism between closed intervals.
\item The PL map projecting $K_j$ onto $\{(r,0) \mid -\frac{1}{2} \leq r \leq 1\}$ canoncially.
\end{enumerate}   
\end{enumerate}
\end{enumerate}
\item
\label{prop:2.4}
Let $M_f$ be the set of all points except points in the disjoint union ${\sqcup}_{j^{\prime}} C_{j^{\prime}}$ where $C_{j^{\prime}}$ is as in {\rm (\ref{prop:2.3.3})} and $M_f \subset W_f$ is the complementary set of the set of all points which are so-called {\rm non-manifold} points. $M_f$ is also regarded as a smooth manifold and the restriction of $\bar{f}$ there is also a smooth immersion. Smooth manifolds or submanifolds of $W_f$ appearing in {\rm (\ref{prop:2.1})} and {\rm (\ref{prop:2.2})} are regarded as submanifolds of this if they are subsets of $M_f$.
\end{enumerate}
\end{Prop}

We do not need to understand this rigorous and lengthy proposition well in our paper. We can know the proof of this result by investigating section 2 of \cite{kobayashisaeki} for example. There, it is shown that for a stable fold map $f$ on a closed and smooth manifold whose dimension is greater than $2$, $W_f$ is a $2$-dimensional polyhedron having local topologies of finitely many certain types. More precisely, it is also shown for a {\it stable} map on such a manifold into a surface with no boundary. Note that stable maps exist densely in this case under the topologies of the sets of smooth maps between two smooth manifolds are as presented just after Fact \ref{fact:2}. In a word, they are so-called {\it Whitney $C^{\infty}$ topologies}.

The following is a fundamental principle for construction. In the present paper we concentrate on
the case $(m^{\prime},n)=(3,2)$. We use this and the construction of local smooth maps in (the sketch of) its proof implicitly or explicitly throughout the present paper.
\begin{Prop}
\label{prop:3}
For a simple fold map $f$ in {\rm Definition \ref{def:1}} with $m>n \geq 1$ and for any integer $m^{\prime}>n$, there exist an $m$-dimensional closed and smooth manifold $M^{\prime}$ and a simple fold map $f^{\prime}:M^{\prime} \rightarrow N$ enjoying the following properties.
\begin{enumerate}
\item
\label{prop:3.1}
 There exists a PL homeomorphism $\phi:W_f \rightarrow W_{f^{\prime}}$ satisfying the relation ${\bar{f}}^{\prime} \circ \phi=\bar{f}$ and compatible with the remaining properties.
\item
\label{prop:3.2}
$q_{f^{\prime}}(S(f^{\prime}))$ is the disjoint union of the set of all non-manifold points in $W_f$ and points in the boundary $\partial M_f$ of the $n$-dimensional compact manifold $M_f$. For points in the former case the indices of the singular points in the preimages are $1$ and they are $0$ in the latter case.

\item
\label{prop:3.3}
 $\phi$ maps the set ${\rm Int}\ M_f$, which is a smooth manifold and smooth submanifold of $M_f \supset {\rm Int}\ M_f$, onto the complementary set of $q_{f^{\prime}}(S(f^{\prime})) \subset W_{f^{\prime}}$ as a {\rm (}PL{\rm )} homeomorphism and a diffeomorphism. 
 \item
 \label{prop:3.4}
  $\phi$ maps the set $W_f-{\rm Int}\ M_f$, which is the disjoint union of the boundary $\partial M_f \subset M_f$ of the manifold $M_f$ and the set of all non-manifold points of $W_f$, onto $q_{f^{\prime}}(S(f^{\prime}))$ as a {\rm (}PL{\rm )} homeomorphism and a diffeomorphism.
 \item
 \label{prop:3.5}
  $\phi$ maps the closure of each connected component of $M_f$, which is also a smooth manifold, onto an $n$-dimensional smooth manifold embedded in $W_{f^{\prime}}$, as {\rm PL} homeomorphism and a diffeomorphism. 
\item
\label{prop:3.6}
 For $f^{\prime}$, preimages containing no singular points are disjoint unions of copies of the {\rm (}$m^{\prime}-n{\rm )}$-dimensional unit sphere $S^{m^{\prime}-n}$.
\end{enumerate}

\end{Prop}

It is regarded as a fundamental property of (simple) fold maps. \cite{kitazawa} also presents a proof of Proposition \ref{prop:3} for example. We also give a sketch of a proof.
\begin{proof}[A sketch of a proof of Proposition \ref{prop:3}.]

We construct a local map onto a small regular neighborhood $N(C)$ of $C \subset q_f(S(f))$ such that $N(C)$ is regarded as the total space of a trivial PL bundle over $C$ whose fiber is a closed interval $[0,1]$ and that gives the identification of $C \subset q_f(S(f))$ with $C \times \{0\} \subset N(C)$ via the map mapping $x$ to $(x,0)$ in Proposition \ref{prop:2}.

We construct a Morse function $r_{c}$ on a copy of the unit disk $D^{m-1}$ defined by $r(x_1,\cdots,x_{m-1}):={\Sigma}_{j=1}^{m-1} {x_j}^2+c$ where $c$ is a suitable real number. We construct the product map of $r_{c}$ and the identity map on $C$. By composing a suitable PL homeomorphism, we have a local map onto $N(C)$.

We construct a local map onto a small regular neighborhood $N(C)$ of $C \subset q_f(S(f))$ such that $N(C)$ is regarded as the total space of a PL bundle over $C$ whose fiber is $K:=\{(r \cos t, r \sin t) \mid 0 \leq r \leq 1, t=0,\frac{2}{3}\pi, \frac{4}{3}\pi.\} \subset {\mathbb{R}}^2$ and that gives the identification of $C$ with $C \times \{(0,0)\} \subset N(C)$ via the map mapping $x$ to $(x,(0,0))$ in Proposition \ref{prop:2}.

We have a Morse function ${r_{c}}^{\prime}$ where $c$ is a suitable real number. 

\begin{itemize}
\item

 The manifold of the domain is a manifold obtained by removing the interiors of three smoothly and disjointly embedded copies of the ($m-1$)-dimensional unit disk $D^{m-1}$ from the ($m-1$)-dimensional unit sphere $S^{m-1}$.
\item

 ${r_{c}}^{\prime}$ has exactly one singular point and it is in the interior.
\item 
Around this singular point, there exists a suitable coordinate and we have ${r_{c}}^{\prime}(x_1,\cdots,x_{m-1})={\Sigma}_{j=1}^{m-2} {x_j}^2-{x_{m-1}}^2+c$ for the coordinate.
\item 

The preimage of the minimal value is one or the disjoint union of two of the connected components of the boundary of the manifold.
\item 

The preimage of the maximal value is the disjoint union of the remaining connected components of the boundary of the manifold.
\end{itemize}
We show the construction only in the case where the total space of the PL bundle over $C$ whose fiber is $K$ is trivial. 
We construct the product map of $r_{c}$ and the identity map on $C$. By composing a suitable PL homeomorphism, we have a local map onto $N(C)$.
We can also construct a map in the case where the bundle is non-trivial as presented in Proposition \ref{prop:2} (\ref{prop:2.3.3.1}) by twisting the trivial smooth family of the Morse functions.
Consider the disjoint union of all $N(C)$ here and its interior. Over the complementary set of this in $W_f$, we can construct the projection of a trivial smooth bundle whose fiber is diffeomorphic to $S^{m-2}$. By constructing this, we have all desired local maps and have a surjective map onto $W_f$ by gluing them in a canonical way.

Last we compose the resulting surjective map onto $W_f$ with $\bar{f}$. This is a desired simple fold map.
 
\end{proof}
\begin{Def}[Definition \ref{def:4}]
\label{def:5}
In Proposition \ref{prop:2} or \ref{prop:3}, the bundles whose fibers are $K$ are trivial, then $f$ or $f^{\prime}$ is said to be {\it normal} and $W_f$ or $W_{f^{\prime}}$ is also said to be {\it normal}. 
Hereafter, we say the map $f$ is {\it normal} if the restriction of the map to $f^{-1}(N(C))$ is the product map of a Morse function and the identity map on $C$ for suitable coordinates respecting the structures of the products in the (sketch of the) proof of Proposition \ref{prop:3}.
\end{Def}

\begin{Rem}
	\label{rem:0}
	In Definition \ref{def:5}, from some well-known arguments on singularity theory and the groups of homeomorphisms and diffeomorphisms of conpact surfaces, the new definition of a simple fold map is satisfied only from the definition in Definition \ref{def:4} for example in the case where simple fold maps are from $3$-dimensional closed and orientable manifolds into orientable surfaces with no boundaries. Respecting this, after Definition \ref{def:10}, in the third section, the class of a {\it most normal} simple fold maps is defined as a subclass of this class. However, in the case of such $3$-dimensional manifolds and maps, normal simple fold maps are always most normal. This is also fundamental facts in \cite{saeki2}, followed by \cite{kitazawasaeki} as a closely related study, for example.
\end{Rem}

For $m>n=2$, the Reeb space is a so-called {\it simple polyhedron} with no {\it vertices}.
 \cite{ikeda} is one of pioneering works on topological theory of such polyhedra.

  \cite{turaev} is also on a related study and presents a study on so-called {\it shadows} of $3$-dimensional manifolds. \cite{costantinothurston, ishikawakoda} are related to both {\it shadows} and stable maps on $3$-dimensional closed manifolds into surfaces with no boundaries. 

Hereafter, we essentially concentrate on normal $2$-dimensional polyhedra, which are the Reeb spaces of some simple fold maps into some surfaces with no boundaries. However, in considerable situations, we encounter $2$-dimensional such polyhedra which are not normal. Section 7 of \cite{kobayashisaeki} presents good examples. More precisely, FIGURE 7 (c) there is one of such examples.
 In studies presented here, we encounter such polyhedra in general. We can also encounter $2$-dimensional polyhedra which are locally the Reeb spaces of some simple fold maps and which are not globally in these studies of shadows for example.
\subsection{Special generic maps.}

This subsection is a kind of appendices. 

\begin{Def}
\label{def:6}
A fold map $f$ is said to be {\it special generic} if $i(p)=0$ for any $p \in S(f)$.
\end{Def}

\begin{Prop}
\label{prop:4}
Special generic maps are simple.
\end{Prop}

For example, canonical projections of unit spheres, which are also round, are special generic.

We do not concentrate on special generic maps essentially. We comment on Propositions \ref{prop:2} and \ref{prop:3} again by presenting Fact \ref{fact:3}.

\begin{Fact}
\label{fact:3}
In {\rm Propositions \ref{prop:2}} and {\rm \ref{prop:3}}, if $M_f=W_f$, then we obtain $f^{\prime}$ as a special generic map{\rm :} for a smooth immersion of an $n$-dimensional compact and smooth manifold into $N$, we can construct a special generic map $f^{\prime}:M^{\prime} \rightarrow N$ into $N$ on a suitable $m^{\prime}$-dimensional closed and smooth manifold $M^{\prime}$ such that $\bar{f^{\prime}}$ and the original immersion agree.
\end{Fact}

For this subsection, see also \cite{saeki1} for example.
 
 \subsection{Graph manifolds and graphs we can define for graph manifolds.}
 
 \begin{Def}
 	\label{def:7}
 	A $3$-dimensional closed and orientable manifold $M$ is said to be a {\it graph} manifold if there exist finitely many disjoint tori smoothly embedded in $M$ and by removing the interiors of suitable small regular neighborhoods of them, we have finitely many (total spaces of) smooth bundles over compact and connected surfaces whose fibers are circles.
 \end{Def}
 A {\it pair of pants} is a $2$-dimensional manifold obtained by removing the interiors of three disjointly and smoothly embedded copies of the $2$-dimensional unit disk $D^2$ from a copy of the $2$-dimensional unit sphere $S^2$.
 \begin{Prop}[E. g. \cite{kitazawasaeki}]
 	\label{prop:5}
 	In {\rm Definition \ref{def:7}}, we can choose the tori in such a way that the resulting base spaces of the bundles are copies of the $2$-dimensional unit disk $D^2$ or pairs of pants and that the bundles are trivial smooth bundles. Furthermore, for each of the small regular neighborhoods of these tori, we can also do so that for the distinct connected components of the boundary of the regular neighborhood of each torus are in the boundaries of the distinct {\rm (}total spaces of the{\rm }) bundles over a copy of $D^2$ or a pair of pants.
 \end{Prop}
 The following graph is used in \cite{kitazawasaeki} where we do not define the notion rigorously or explicitly.
 \begin{Def}
 	\label{def:8}
 	A {\it representation graph} for a graph manifold $M$ is a graph enjoying the following two properties.
 	\begin{enumerate}
 		\item The vertex set consists of all bundles obtained after removing the interiors of the regular neighborhoods of the tori in {\rm Proposition \ref{prop:5}}.
 		\item The edge set consists of all tori here. The two vertices in an edge $e$ or a torus are the bundles or vertices in which the connected components of the boundary of the regular neighborhood of the torus $e$ are.   
 	\end{enumerate}
 \end{Def} 
 Such a graph is not unique for a graph manifold. This graph has no loops. This graph may be a multigraph. This graph is a finite graph. This graph is connected if and only if $M$ is connected. The degree of a vertex is $1$ if the base space is the $2$-dimensional unit disk $D^2$ and $3$ if the base space is a pair of pants.
 In \cite{neumann}, similar graphs are introduced and studied. However, our graphs are different from the graphs. However, both kinds of graphs inherit similar important topological properties of the original manifolds in considerable situations. We will also introduce these graphs in \cite{neumann} shortly in the end of our paper.
\subsection{Simple fold maps on $3$-dimensional closed and orientable manifolds and graph manifolds into ${\mathbb{R}}^2$.}
It is well-known that $3$-dimensional closed manifolds admit (stable) fold maps into ${\mathbb{R}}^2$. See \cite{levine} and also \cite{thom, whitney}.
Moreover, the following theorem or Theorem \ref{thm:1} has been shown for (simple and stable) fold maps on graph manifolds into ${\mathbb{R}}^2$.

As presented before Main Theorem \ref{mthm:1} in the first section, for a round fold map into ${\mathbb{R}}^n$, we have a PL bundle over the ($n-1$)-dimensional unit sphere $S^{n-1}$
by considering the projection defined as the map 
from ${({\phi}_{{\mathbb{R}}^n} \circ \bar{f})}^{-1}(\{x \in {\mathbb{R}}^n \mid ||x|| \geq \frac{1}{2} \})$ onto $S^{n-1}$ mapping $x$ to $\frac{1}{||{({\phi}_{{\mathbb{R}}^n} \circ \bar{f})}(x)||} {({\phi}_{{\mathbb{R}}^n} \circ \bar{f})}(x)$ in the situation of Definition \ref{def:2}. Its fiber is a $1$-dimensional polyhedron. 
\begin{Def}[Definition \ref{def:4}]
	\label{def:9}
	A {\it topologically quasi-trivial} round fold map means a round fold map such that a PL bundle obtained in this way is trivial for a suitable diffeomorphism ${\phi}_{{\mathbb{R}}^n}:{\mathbb{R}}^n \rightarrow {\mathbb{R}}^n$ in the situation of Definition \ref{def:2}.
\end{Def}
\begin{Thm}
\label{thm:1}
\begin{enumerate}
\item \label{thm:1.1}
{\rm (}\cite{saeki2}.{\rm )} A $3$-dimensional closed and orientable manifold admits a normal simple fold map into ${\mathbb{R}}^2$ if and only if it is a graph manifold.
\item \label{thm:1.2}
{\rm (}\cite{saeki2}.{\rm )} A $3$-dimensional closed and orientable manifold admits a normal strongly simple fold map into ${\mathbb{R}}^2$ if and only if it is a graph manifold.
\item \label{thm:1.3}
{\rm (}\cite{kitazawasaeki}{\rm )} A $3$-dimensional closed and orientable manifold admits a topologically quasi-trivial round fold map into ${\mathbb{R}}^2$ if and only if it is a graph manifold.
\end{enumerate}
\end{Thm}

Note that conditions such as ones given by words "normal" and "topologically quasi-trivial" are not presented explicitly there. We can add these conditions by the construction of the maps. Related to Remark \ref{rem:0}, we can see that a topologically quasi-trivial round fold map from a $3$-dimensional closed and orientable manifold into ${\mathbb{R}}^2$ is also normal (in the sense of Definition \ref{def:5}). We show a sketch of a proof of Theorem \ref{thm:1} (\ref{thm:1.3}) in the next section.

\section{Proofs of Main Theorems.}
We prove Main Theorems. 

\begin{Def}
\label{def:10}
We call the following two Morse functions {\it most standard} Morse functions.
\begin{enumerate}
\item A Morse function on a copy of a unit disk satisfying the following two.
\begin{enumerate}
\item The boundary is the preimage of the minimal (maximal) value.
\item There exists exactly one singular point, at which the function has the maximal (resp. minimal) value.
\end{enumerate}
\item A Morse function on a pair of pants satisfying the following three.
\begin{enumerate}
\item The disjoint union of two of the connected components of the boundary is the preimage of the minimal (maximal) value.
\item The remaining connected component of the boundary is the preimage of the maximal (resp. minimal) value.
\item There exists exactly one singular point and it is in the interior of the pair of pants. 
\end{enumerate}
\end{enumerate}
\end{Def}
These functions are regarded as local functions around each singular point of simple fold
maps into surfaces constructed in (the sketch of the proof of) Proposition \ref{prop:3} and they are natural Morse functions (for general $m \geq 3$). 
Hereafter, we define a {\it most normal} simple fold map $f$ as a normal simple fold map, first defined in the end of Definition \ref{def:5}, such that the restriction to $f^{-1}(N(C))$ is regarded as the product map of a most standard Morse function and the identity map on $C$ for each connected component $C \subset q_f(S(f))$ of the image $q_f(S(f))$ of the singular set $S(f)$ and a suitable small regular neighborhood $N(C)$, which is regarded as a trivial PL bundle over $C$ whose fiber is a
$1$-dimensional polyhedron. Related to Remark \ref{rem:0}, in the case of simple fold maps on $3$-dimensional closed and orientable manifolds into orientable surfaces with no boundaries, normal simple fold maps are most normal. 

We introduce some important fold maps and the restrictions of these fold maps to suitable compact and smooth submanifolds into ${\mathbb{R}}^2$.

\begin{Prop}
\label{prop:6}
For any integer $m \geq 3$ and an integer $i=1,2$, a copy of the unit sphere $S^m$ admits a most normal quasi-trivial round fold map $f$ into ${\mathbb{R}}^2$ such that the following conditions hold.
\begin{enumerate}
\item We can do so that ${\phi}_{{\mathbb{R}}^2}(f(S(f)))=\{x \in {\mathbb{R}}^2 \mid 1 \leq ||x|| \leq 3, ||x|| \in \mathbb{N}\}$ holds in {\rm Definition \ref{def:2}}.
\item ${\phi}_{{\mathbb{R}}^2}(f(F_0(f)))=\{x \in {\mathbb{R}}^2 \mid ||x||=i,3\}$.
\item ${\phi}_{{\mathbb{R}}^2}(f(F_1(f)))=\{x \in {\mathbb{R}}^2 \mid ||x||=3-i\}$.
\end{enumerate}
Note that here in the case $m=3$, we do not need the assumption that the map is {\rm (}most{\rm )} normal, related to Remark \ref{rem:0}.

\end{Prop}
FIGURE 7 (b) shows the Reeb space of such a map.
\begin{Def}
\label{def:11}
In Proposition \ref{prop:6}, if we restrict $f$ to the submanifold obtained by removing the connected component of the preimage of $\{x \in {\mathbb{R}}^2 \mid i-\frac{1}{2}<||x||< i+\frac{1}{2}\}$ containing a connected component of $({\phi}_{{\mathbb{R}}^2} \circ f)(F_0(f))$ for ${\phi}_{{\mathbb{R}}^2} \circ f$, then we have a smooth map and we call this a {\it bordered doubled most standard} map {\it induced from $S^m$ into the plane}. Furthermore, we call the images of connected components of the boundary for the quotient map to the Reeb space of the resulting map the {\it borders}.
\end{Def}
We immediately have the following proposition by fundamental arguments on Morse functions and fold maps.
\begin{Prop}
\label{prop:7}
In the situation of Definition \ref{def:11},
for $i=1,2$, from a map $f$ in Proposition \ref{prop:6}, we always have a map in {\rm Definition \ref{def:11}} on the product of a copy of the $2$-dimensional unit disk $D^2$ and a copy of the {\rm (}$m-2${\rm )}-dimensional unit sphere $S^{m-2}$.
\end{Prop}

Hereafter, we concentrate on smooth maps on $3$-dimensional manifolds into ${\mathbb{R}}^2$ where we can generalize for $m$-dimensional manifolds similarly for general $m \geq 3$.

We introduce another map. For each point in $x \in {\mathbb{R}}^k$, we naturally regard it as a vector here and we consider the vector $x_1-x_2$ for $x_1 \in {\mathbb{R}}^k$ and $x_2 \in {\mathbb{R}}^k$ defined by considering the difference in the vector space. For each vector $x$, $||x||$ is the value of the Euclidean norm at $x$, which is same as $||x||$ presented in the introduction of the first section, denoting the distance of $x$ and the origin $0$ under the standard Euclidean metric.

\begin{Def}
\label{def:12}
Let $m \geq 3$ be an integer and $(i_1,i_2,i_3)$ and $({i_1}^{\prime},{i_2}^{\prime},{i_3}^{\prime})$ be triplets of integers $1$ or $2$. Let $f$ be a most normal strongly simple fold map into ${\mathbb{R}}^2$ such that the following properties are enjoyed.
\begin{enumerate}
\item By a suitable diffeomorphism ${{\phi}^{\prime}}_{{\mathbb{R}}^2}$ on ${\mathbb{R}}^2$, ${{\phi}^{\prime}}_{{\mathbb{R}}^2}(f(S(f)))$ is the disjoint union of the following sets.
\begin{enumerate}
\item $\{x \in {\mathbb{R}}^2 \mid 1 \leq ||x-(0,0)|| \leq 3, ||x-(0,0)|| \in \mathbb{N}\}$.
\item $\{x \in {\mathbb{R}}^2 \mid 1 \leq ||x-(-7,0)|| \leq 3, ||x-(-7,0)|| \in \mathbb{N}\}$.
\item $\{x \in {\mathbb{R}}^2 \mid 1 \leq ||x-(7,0)|| \leq 3, ||x-(7,0)|| \in \mathbb{N}\}$.
\item $\{x \in {\mathbb{R}}^2 \mid ||x||=11,12\}$.
\end{enumerate}
\item ${{\phi}^{\prime}}_{{\mathbb{R}}^2}(f(F_0(f)))$ is the disjoint union of the following sets.
\begin{enumerate}
\item $\{x \in {\mathbb{R}}^2 \mid ||x-(0,0)||=1,i_1+1\}$.
\item $\{x \in {\mathbb{R}}^2 \mid ||x-(-7,0)||=1,i_2+1\}$.
\item $\{x \in {\mathbb{R}}^2 \mid ||x-(7,0)||=1,i_3+1\}$.
\item $\{x \in {\mathbb{R}}^2 \mid ||x||=12\}$.
\end{enumerate}
\item ${{\phi}^{\prime}}_{{\mathbb{R}}^2}(f(F_1(f)))$ is the disjoint union of the following sets.
\begin{enumerate}
\item $\{x \in {\mathbb{R}}^2 \mid ||x-(0,0)||=4-i_1\}$.
\item $\{x \in {\mathbb{R}}^2 \mid ||x-(-7,0)||=4-i_2\}$.
\item $\{x \in {\mathbb{R}}^2 \mid ||x-(7,0)||=4-i_3\}$.
\item $\{x \in {\mathbb{R}}^2 \mid ||x||=11\}$.
\end{enumerate}
\item The Reeb space $W_f$ and a triplet  $({i_1}^{\prime},{i_2}^{\prime},{i_3}^{\prime})$ of integers which are $1$ or $2$ enjoy the following properties.
\begin{enumerate}
\item $W_f-{({{\phi}^{\prime}}_{{\mathbb{R}}^2} \circ \bar{f})}^{-1}(\{x \in {\mathbb{R}}^2 \mid 11 \leq ||x|| \leq 12\})$ consists of exactly two connected components $W_{f,1}$ and $W_{f,2}$.
\item ${({{\phi}^{\prime}}_{{\mathbb{R}}^2} \circ \bar{f})}^{-1}(\{x \in {\mathbb{R}}^2 \mid ||x-(-7,0)|| <1\})$ and ${({{\phi}^{\prime}}_{{\mathbb{R}}^2} \circ \bar{f})}^{-1}(\{x \in {\mathbb{R}}^2 \mid ||x-(0,0)||<1\})$ are connected and in a same connected component of the two connected components $W_{f,1}$ and $W_{f,2}$ if and only if ${i_1}^{\prime}=1$.
\item ${({{\phi}^{\prime}}_{{\mathbb{R}}^2} \circ \bar{f})}^{-1}(\{x \in {\mathbb{R}}^2 \mid ||x-(0,0)||<1\})$ and ${({{\phi}^{\prime}}_{{\mathbb{R}}^2} \circ \bar{f})}^{-1}(\{x \in {\mathbb{R}}^2 \mid ||x-(7,0)||<1\})$ are connected and in a same connected component of the two connected components $W_{f,1}$ and $W_{f,2}$ if and only if ${i_2}^{\prime}=1$.
\item ${({{\phi}^{\prime}}_{{\mathbb{R}}^2} \circ \bar{f})}^{-1}(\{x \in {\mathbb{R}}^2 \mid ||x-(-7,0)||<1\})$ and ${({{\phi}^{\prime}}_{{\mathbb{R}}^2} \circ \bar{f})}^{-1}(\{x \in {\mathbb{R}}^2 \mid ||x-(7,0)||<1\})$ are connected and in a same connected component of the two connected components $W_{f,1}$ and $W_{f,2}$ if and only if ${i_3}^{\prime}=1$.
\end{enumerate}
\end{enumerate}
If we restrict $f$ to the submanifold obtained by removing the connected components of the preimages of the following open sets containing connected components of $F_0(f)$ for ${{\phi}^{\prime}}_{{\mathbb{R}}^2} \circ f$, then we have a smooth map and we call this a {\it bordered doubled most standard} map {\it supporting a pair of pants}.
\begin{itemize}
\item $\{x \in {\mathbb{R}}^2 \mid i_1+\frac{1}{2}<||x-(0,0)||<i_1+\frac{3}{2}\}$.
\item $\{x \in {\mathbb{R}}^2 \mid i_2+\frac{1}{2}<||x-(-7,0)||<i_2+\frac{3}{2}\}$.
\item $\{x \in {\mathbb{R}}^2 \mid i_3+\frac{1}{2}<||x-(7,0)||<i_3+\frac{3}{2}\}$.
\end{itemize}
Furthermore, we call the images of connected components of the boundary for the quotient map to the Reeb space of the resulting map the {\it borders}. 
\end{Def}
We review construction of a round fold map on a copy of $S^2 \times S^1$ in \cite{kitazawa0.1, kitazawa0.2, kitazawa0.4, kitazawa0.5} for example. We argue using a representation graph. Such expositions are not in these papers. However, essentially we argue similarly.
This is the total space of a trivial smooth bundle over $S^2$ whose fiber is a circle. The base space $S^2$ is decomposed into two copies of the $2$-dimensional unit sphere and a manifold diffeomorphic to $S^1 \times D^1$ by two circles smoothly and disjointly embedded in $S^2$. The last manifold, diffeomorphic to $S^1 \times D^1$ is, divided into a copy of the $2$-dimensional unit disk and a pair of pants by a circle smoothly embedded in the interior of the manifold, diffeomorphic to $S^1 \times D^1$. $S^2 \times S^1$ is thus divided into four trivial bundles. This means that this manifold has a representation graph which is connected and has exactly three vertices of degree $1$, exactly one vertex of degree $3$ and exactly four edges. We can construct smooth maps as the projections of the given smooth trivial bundles over copies of $D^2$ on the first two manifolds. 
We set this base space as the disk consisting of all points $x \in {\mathbb{R}}^2$ satisfying $||x|| \leq \frac{21}{2}$. We construct the product map of a most standard Morse function and the identity map on $\{x \in {\mathbb{R}}^2 \mid ||x||=11\}$ onto $\{x \in {\mathbb{R}}^2 \mid \frac{21}{2} \leq ||x|| \leq \frac{23}{2}\}$ on the third manifold and the product map of a most standard Morse function and the identity map on $\{x \in {\mathbb{R}}^2 \mid ||x||=12\}$ onto $\{x \in {\mathbb{R}}^2 \mid \frac{23}{2} \leq ||x|| \leq 12\}$ on the fourth manifold. For these Morse functions here, indices of singular points are $1$ and $0$, respectively. Thus we have local maps on four total spaces of bundles or vertices. By the situation making the identifications between the boundaries of these four manifolds sufficiently simple, we may skip arguments on construction of local maps on the small regular neighborhoods of tori or edges of the graph. Our construction here also gives local maps on the small regular neighborhoods of tori or edges of the graph. In short, we have obtained a desired round fold map on $S^2 \times S^1$ into ${\mathbb{R}}^2$. 

We can apply this to the total space of a general smooth bundle over $S^2$ whose fiber is a circle or a higher dimensional general unit sphere.
This is closely related to a main ingredient of construction of simple fold maps in \cite{saeki2}, followed by \cite{kitazawasaeki}, and related arguments are applied in our proof of Main Theorem \ref{mthm:1} for example.

For the resulting round fold map on $S^2 \times S^1$, we remove a suitable connected component of
 the preimage of each of the following three open subsets of ${\mathbb{R}}^2$.
\begin{itemize}
\item $\{x \in {\mathbb{R}}^2 \mid ||x-(0,0)||< \frac{7}{2}\}$.
\item $\{x \in {\mathbb{R}}^2 \mid ||x-(-7,0)||< \frac{7}{2}\}$.
\item $\{x \in {\mathbb{R}}^2 \mid ||x-(7,0)||< \frac{7}{2}\}$.
\end{itemize}
By constructing (local) suitable product maps as in the sketch of our proof of Proposition \ref{prop:3} or the product maps of most standard Morse functions and the identity map on circles instead and attaching them in a suitable way, we have a new strongly simple fold map $f$ on some new $3$-dimensional closed, connected and orientable manifold with the following proposition.
\begin{Prop}
\label{prop:8}
In the situation of Definition \ref{def:12},
for any triplet $(i_1,i_2,i_3)$ and any triplet $({i_1}^{\prime},{i_2}^{\prime},{i_3}^{\prime})$,
we can have a {\rm (}most{\rm )} normal strongly simple fold map $f:M \rightarrow {\mathbb{R}}^2$ on some new $3$-dimensional closed, connected and orientable manifold $M$ enjoying the presented properties. Furthermore, we can have such a map $f$ such that from this we have a bordered doubled most standard map supporting a pair of pants on the product of a pair of pants and a circle in the presented way.
\end{Prop}
FIGURE \ref{fig:1} presents the Reeb space of such a map.
\begin{figure}
\includegraphics[width=50mm]{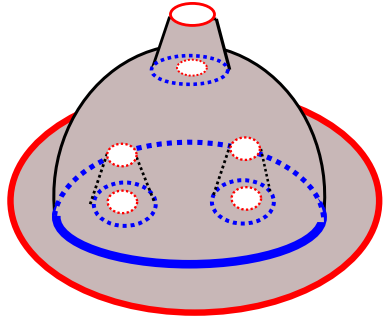}
\caption{The Reeb space of a map $f$ for $(i_1,i_2,i_3)=(1,1,1)$ and $({i_1}^{\prime},{i_2}^{\prime},{i_3}^{\prime})=(2,2,1)$ in Proposition \ref{prop:8}. Red (blue) colored circles are for the values at singular points whose indices are 0 (resp. 1).}
\label{fig:1}
\end{figure}
\begin{Def}
	\label{def:13}
	For an integer $m \geq 3$, a pair $(i_1,i_2)$ of integers $0$ or $2$, and an integer $l \geq i_1+i_2+3$, let $f$ be a most normal and topologically quasi-trivial round fold map on an $m$-dimensional closed and connected manifold into ${\mathbb{R}}^2$ such that the following conditions hold.
	\begin{enumerate}
		\item There exists a diffeomorphism ${\phi}_{{\mathbb{R}}^2}$ satisfying ${\phi}_{{\mathbb{R}}^2}(f(S(f)))=\{x \in {\mathbb{R}}^2 \mid 1 \leq ||x|| \leq l, ||x|| \in \mathbb{N}\}$ as in Definition \ref{def:2}.
		\item ${\phi}_{{\mathbb{R}}^2}(f(F_0(f)))=\{x \in {\mathbb{R}}^2 \mid ||x||=1,1+i_1,l-i_2,l\}$.
		\item ${\phi}_{{\mathbb{R}}^2}(f(F_1(f)))={\phi}_{{\mathbb{R}}^2}(f(S(f)-F_0(f)))$.
		\item For ${\phi}_{{\mathbb{R}}^2} \circ f$ preimages are as follows.
		\begin{enumerate}
			\item The preimage of a point in $\{x \in {\mathbb{R}}^2 \mid l-i_2-1<||x||<l-i_2,l-1<||x||<l\}$ is connected.
			\item The preimage of a point in $\{x \in {\mathbb{R}}^2 \mid l-i_2<||x||<l-i_2+1\}$ consists of exactly two connected components if $i_2=2$.
			\item The preimage of a point in $\{x \in {\mathbb{R}}^2 \mid l-i_2-j-1<||x||<l-i_2-j\}$ consists of exactly $j+1$ connected components for an integer $j$ satisfying $0 \leq j$ and $l-i_2-j-1 \geq 1+i_1$.
			\item The preimage of a point in $\{x \in {\mathbb{R}}^2 \mid i_1<||x||<i_1+1\}$ consists of exactly $l-i_1-i_2$ connected components and that of a point in $\{x \in {\mathbb{R}}^2 \mid 1<||x||<i_1\}$ consists of exactly $l-i_1-i_2-1$ connected components if $i_1=2$.
			\item The preimage of a point in $\{x \in {\mathbb{R}}^2 \mid 0<||x||<1\}$ consists of exactly $l-i_1-i_2-2$ connected components.
		\end{enumerate}
	\end{enumerate}
	If we restrict $f$ to the submanifold obtained by removing the connected components of the preimages of the two following sets containing some connected components of $F_0(f)$ for ${\phi}_{{\mathbb{R}}^2} \circ f$, then we have a smooth map and we call this an {\it S-map supporting an annulus} of {\it type $(i_1,i_2,l,{\epsilon}_1,{\epsilon}_2)$} where ${\epsilon}_i$ is defined in the following way.
	\begin{itemize}
		\item $\{x \in {\mathbb{R}}^2 \mid \frac{1}{2}<||x||<\frac{3}{2}\}$ if $i_1=0$ and either $\{x \in {\mathbb{R}}^2 \mid \frac{1}{2}<||x||<\frac{3}{2}\}$ or $\{x \in {\mathbb{R}}^2 \mid \frac{5}{2}<||x||<\frac{7}{2}\}$ if $i_1=2$. In the case $i_1=0$, we put ${\epsilon}_1:=0$ and in the case $i_1=2$, we put ${\epsilon}_1:=\pm1$ if we remove the connected component of the preimage of $\{x \in {\mathbb{R}}^2 \mid \frac{3}{2}+{\epsilon}_1<||x||<\frac{5}{2}+{\epsilon}_1\}$.
		\item $\{x \in {\mathbb{R}}^2 \mid l-\frac{1}{2}<||x||<l+\frac{1}{2}\}$ if $i_2=0$ and either $\{x \in {\mathbb{R}}^2 \mid l-\frac{1}{2}<||x||<l+\frac{1}{2}\}$ or $\{x \in {\mathbb{R}}^2 \mid l-\frac{5}{2}<||x||<l-\frac{3}{2}\}$ if $i_2=2$. In the case $i_1=0$, we put ${\epsilon}_2:=0$ and in the case $i_1=2$, we put ${\epsilon}_2:=\pm1$ if we remove the connected component of the preimage of $\{x \in {\mathbb{R}}^2 \mid l-\frac{3}{2}-{\epsilon}_2<||x||<l-\frac{1}{2}-{\epsilon}_2\}$.
	\end{itemize}
	Furthermore, we call the images of connected components of the boundary for the quotient map to the Reeb space of the resulting map the {\it borders}.
\end{Def}
An {\it S-map} is first defined in \cite{saeki2}. S-maps supporting annuli are defined here first and regarded as specific S-maps. We present the following theorem as a theorem essentially shown in \cite{saeki2} or a kind of exercises related to the article.

\begin{Thm}[E.g. \cite{saeki2}]
	\label{thm:2}
		\begin{enumerate}
			\item \label{thm:2.1}
			\begin{enumerate}
				\item We have a representation graph collapsing to a point for the manifold of the domain of a map in Definition \ref{def:13}. 
                \item The Reeb space $W_f$ of a map $f$ in Definition \ref{def:13} collapses to a point or a bouquet of finitely many copies of the $2$-dimensional unit sphere $S^2$
			  \item Furthermore, if and only if $l=i_1+i_2+3$ holds, then the Reeb space $W_f$ of a map $f$ in Definition \ref{def:13} collapses to a point.
			  \end{enumerate}
			\item \label{thm:2.2} 
	A manifold diffeomorphic to $S^1 \times S^1 \times D^1$ admits an S-map supporting an annuls of any type $(i_1,i_2,l,{\epsilon}_1,{\epsilon}_2)$ satisfying the assumption in Definition \ref{def:13}.
	\end{enumerate} 
\end{Thm}

For the Reeb space of such a map, see FIGURE 3 of \cite{kitazawasaeki} and see also \cite{saeki2} for example.

Hereafter, for S-maps supporting annuli, we essentially consider the case $l=i_1+i_2+3$ in Definition \ref{def:13}. This respects main ingredients of \cite{kitazawasaeki,saeki2} with \cite{neumann}.

We present a sketch of a proof of Theorem \ref{thm:1} (\ref{thm:1.3}). See \cite{kitazawasaeki} for more rigorous expositions. Main tools and ideas are as follows and they are also important in construction in the proofs of our Main Theorems.
\begin{itemize}
\item Representation graphs of a certain nice type for construction of local maps.
\item Construction of local maps respecting the graphs via methods in \cite{saeki2}. 
\end{itemize}
\begin{proof}[A sketch of a proof of {\rm Theorem \ref{thm:1} (\ref{thm:1.3})}.]
For each vertex $v_{\lambda}$ or the bundle over a pair of pants or a copy of the $2$-dimensional unit disk in Proposition \ref{prop:5} and Definition \ref{def:8}, we construct the product map of a most standard Morse function and the identity map on a circle and compose this with the product map of diffeomorphisms which is a map onto a small closed tubular neighborhood $\{x \in {\mathbb{R}}^2 \mid a_{\lambda}-\frac{1}{2} \leq ||x|| \leq a_{\lambda}+\frac{1}{2}\}$ of the circle $\{x \in {\mathbb{R}}^2 \mid ||x||=a_{\lambda}\}$. More precisely, we can define the family $\{a_{\lambda} \in \mathbb{N}\}_{\lambda}$ of mutually distinct and finitely many positive integers in such way that the differences of distinct numbers are sufficiently large and construct the local maps for these vertices and we do so.

On each edge or the closed tubular neighborhood of the torus in Definition \ref{def:8}, we construct either of the following two.
Note that according to \cite{saeki2} with \cite{neumann}, we can prepare a representation graph of a certain specific type. More precisely, the type is a mild extension of the so-called {\it plumbing type}. This enables us to do construction of local maps as in \cite{saeki2}.

\begin{itemize}
	\item (The composition of) 
	the product map of the projection of a trivial smooth bundle over a closed interval whose fiber is a circle and 
	the identity map on a circle (with a suitable diffeomorphism).
	
	\item (The composition of) an S-map supporting an annulus of the type $(i_1,i_2,i_1+i_2+3,{\epsilon}_1,{\epsilon}_2)$ (with a suitable diffeomorphism). 
\end{itemize}

This completes the proof. We can also see that the resulting round fold map is (most) normal.
\end{proof}

\begin{proof}[A proof of {\rm Main Theorem \ref{mthm:1}}.]
	We consider a representation graph for $M$ of a nice type as in the sketch of our proof of Theorem \ref{thm:1}. 
	Let $\{X_j\}_{j=1}^a$ denote all vertices where $a>0$ is a suitable integer.
	For each vertex $X_j$ or a trivial bundle $X_j$ over the $2$-dimensional unit disk $D^2$ or a pair of pants, we consider a copy of the $2$-dimensional unit sphere or a pair of pants embedded smoothly into a subspace
	${\mathbb{R}}^2 \times L_j \subset {\mathbb{R}}^2 \times \mathbb{R}$ corresponding to the base space of the bundle. 
	We can take the subspaces according to the following rule and do so.
	
	\begin{itemize}
		\item $L_j$ is a one-point set $\{j\}$ if the base space of the trivial bundle is the $2$-dimensional unit disk or equivalently, the vertex $X_j$ is of degree $1$.
		\item $L_j$ is a closed interval $[j-\frac{1}{4},j+\frac{1}{4}]$ if the base space of the bundle is a pair of pants or equivalently, the vertex $X_j$ is of degree $3$.
	\end{itemize}
	
	We may regard these embedded manifolds as subpolyhedra of Reeb spaces of suitable maps of Proposition \ref{prop:7} or \ref{prop:8} in a natural and suitable way where we consider PL embeddings of the Reeb spaces in suitable and natural ways into ${\mathbb{R}}^3$. We can do so that the original smooth maps into ${\mathbb{R}}^2$ are represented as the compositions of the natural maps onto the polyhedra with the restrictions of the canonical projection to the embedded Reeb spaces to ${\mathbb{R}}^2$ (, followed by a suitable diffeomorphism). We may say that this is also due to the topologies of the Reeb spaces and (expositions on local structures of related fold maps and PL maps in) Propositions \ref{prop:2} and \ref{prop:3} for example.
	
	For each edge or each torus $T_{j^{\prime}}$, we consider the small regular neighborhood. 
	
	We find borders such that the preimages are connected components of the boundary of the small regular neighborhood first. We can consider a suitable PL isotopy of the embeddings of the disjoint union of the Reeb spaces of the maps of Proposition \ref{prop:7} or \ref{prop:8} without changing the value of the third component $x_3$ of $(x_1,x_2,x_3) \in {\mathbb{R}}^2 \times \mathbb{R}$ on each point of the embedded space in such a way that the chosen borders are in the boundary of a smoothly embedded copy of the $3$-dimensional unit disk $D^3$ in the complementary set of the disjoint union of the embedded spaces. Note that for this, we need to choose maps in Proposition \ref{prop:7} or \ref{prop:8} and the embeddings of the Reeb spaces suitably beforehand and that we can do so. 
	More precisely, these Reeb spaces are regarded as spaces obtained by attaching annuli to the interiors of copies of the unit disk or pair of pants via PL homeomorphisms from (connected components of) the boundaries onto circles in the original surfaces.
	
	For this, consult also \cite{kitazawa, kitazawa2}. For example, Main Theorem 5 of \cite{kitazawa} studies a generalized case of a case which is essentially same as the case in our present discussion.  
	
	After this deformation by the PL isotopy, we deform the PL embedding by another suitable PL isotopy if we need. We can glue borders directly via PL homeomorphisms or attach the Reeb space of an S-map supporting an annulus along the borders by suitable PL homeomorphisms in ${\mathbb{R}}^3$ where we can consider this polyhedron as an embedded polyhedron and consider this as a suitable one. This is due to some fundamental observations on the topologies of the Reeb spaces of the S-maps supporting annuli, discussed in Theorem \ref{thm:2}, and such observations on the other maps, discussed in Propositions \ref{prop:7} and \ref{prop:8}. In addition, again, by the topologies of the Reeb spaces and (expositions on local structures of related fold maps and PL maps in) Propositions \ref{prop:2} and \ref{prop:3} for example, we can obtain a surjective continuous (PL) map $q_{\bar{M_{\{j^{\prime}\}}}}$ on our new resulting $3$-dimensional compact manifold $\bar{M_{\{j^{\prime}\}}}$ onto our new $2$-dimensional polyhedron $P_{\{j^{\prime}\}}$ in such a way that the following two properties are enjoyed.
	\begin{itemize}
		\item The composition of the map $q_{\bar{M_{\{j^{\prime}\}}}}$ on the resulting $3$-dimensional compact manifold
		$\bar{M_{\{j^{\prime}\}}}$ onto the resulting polyhedron $P_{\{j^{\prime}\}}$ with the restriction of the canonical projection to ${\mathbb{R}}^2$ to this embedded polyhedron is a smooth map. Let $f_{M_{\{j^{\prime}\}}}$ denote this smooth map.
		\item The restrictions of the map $f_{M_{\{j^{\prime}\}}}$ on the resulting $3$-dimensional compact manifold
		$\bar{M_{\{j^{\prime}\}}}$ to the preimages of the original S-map supporting an annulus and the original maps in Proposition \ref{prop:7} or \ref{prop:8} are regarded as the original maps (where we may need to compose these maps with suitable diffeomorphisms on ${\mathbb{R}}^2$ for the modification).
	\end{itemize}
	In this way, we construct a local map on the small regular neighborhood of a torus $T_{j^{\prime}}$ and $\bar{M_{\{j^{\prime}\}}} \subset M$. This is also due to a main ingredient of construction of simple fold maps on graph manifolds of \cite{saeki2} together with the theory \cite{neumann}, which is also presented in the sketch of our proof of Theorem \ref{thm:1}. Due to this for example, if we choose the embeddings of the Reeb spaces of the maps of Proposition \ref{prop:7} or \ref{prop:8} suitably beforehand, then we can do this for the remaining edges or tori one after another by similar arguments. This presents
	a continuous (PL) map $q_{\bar{M_{J^{\prime}}}}$ on a new resulting $3$-dimensional compact manifold $\bar{M_{J^{\prime}}} \subset M$ onto a new $2$-dimensional polyhedron $P_{{J^{\prime}}}$ at each step where $J^{\prime}$ denotes the set of all chosen $j^{\prime}$ for the torus $T_{j^{\prime}}$ or the edge then. We can finally obtain a desired simple fold map $f:M \rightarrow {\mathbb{R}}^2$ on the given graph manifold $M$ with the quotient map $q_f:M \rightarrow W_f$. Furtheremore, we can obtain this in such that the quotient map $q_f$ is represented as the composition of a PL embedding into ${\mathbb{R}}^3$ with the canonical projection to ${\mathbb{R}}^2$. 
	Related to this argument, compare it to Main Theorem 5 and its proof in \cite{kitazawa} again. 
	This completes the proof.
\end{proof}

The {\it Heegaard genus} of a $3$-dimensional closed, connected and orientable manifold is an important invariant for these manifolds. 

We do not expect much knowledge on this and closely related notions. To known more rigorously, precisely and systematically,
see \cite{hempel} for example. 

Related to this notion, we have the following theorem.

\begin{Thm}[Main Theorem \ref{mthm:2}.]
	\label{thm:3}
	A graph manifold $M$ admits a topologically quasi-trivial round fold map into ${\mathbb{R}}^2$ such that the Reeb space can be embedded into ${\mathbb{R}}^3$ and $S^3$ in the PL category if and only if there exists a representation graph for $M$ which is planar. More generally, we have the following two.
	\begin{enumerate}
		\item 
	    \label{thm:3.1}	
		If a graph manifold $M$ admits a topologically quasi-trivial round fold map into ${\mathbb{R}}^2$ such that the Reeb space can be embedded in a $3$-dimensional closed, connected and orientable manifold whose Heegaard genus is $g>0$ in the PL category, then there exists a representation graph for $M$ which can be embedded in a closed, connected and orientable surface of genus $g$.
		\item
		\label{thm:3.2}
		 Conversely, suppose that there exists a representation graph for $M$ which can be embedded in a closed, connected and orientable surface of genus $g>0$. Then there exists a topologically quasi-trivial round fold map $f:M \rightarrow {\mathbb{R}}^2$ such that in the PL category, the Reeb space $W_f$ can be embedded in some $3$-dimensional closed, connected and orientable manifold obtained by gluing
		 $S^1 \times {\Sigma}_{g,1}$ and $D^2 \times S^1$ between the boundaries by a suitable PL homeomorphism where ${\Sigma}_{g,1}$ denotes a surface obtained by removing the interior of a copy of the $2$-dimensional unit disk smoothly embedded in a closed, connected and orientable surface of genus $g$.
	\end{enumerate}
\end{Thm}
\begin{proof}
We first show the fact presented in Main Theorem \ref{thm:2}.
For a graph manifold for which there exists a representation graph being planar, 
we have a planar representation graph of the plumbing type or the type presented before the sketch of the proof of Theorem \ref{thm:1} (\ref{thm:1.3}). This is due to the related arguments in \cite{saeki2} with \cite{neumann}.

By construction of a round fold map in \cite{kitazawasaeki} or the sketch of the proof of Theorem \ref{thm:1} (\ref{thm:1.3}), we have a normal and topologically quasi-trivial round fold map such that the fiber of the bundle in Definition \ref{def:9} is also a planar graph. 

For this, note also that by the global topology of the Reeb space of an S-map supporting an annulus and Theorem \ref{thm:2} (\ref{thm:2.1}) for example, the new graph of the fiber in Definition \ref{def:9} can be obtained as a graph collapsing to a graph isomorphic to the original representation graph given in the beginning. Note also that in general, it is not isomorphic to the original representation graph. This is also important in similar scenes. 

As a result, the Reeb space can be embedded in ${\mathbb{R}}^3$ in the PL category. Rigorously, in our proof of (\ref{thm:3.2}) later, take $g:=0$ and glue ${\Sigma}_{0,1}$ and $S^1 \times D^2$ in a suitable way to know this.


Conversely, if a topologically quasi-trivial round fold map whose Reeb space can be embedded in ${\mathbb{R}}^3$ in the PL category is given, then by fundamental arguments in \cite{saeki2} with \cite{kitazawasaeki} or (important arguments in) our proof of Theorem \ref{thm:1} (\ref{thm:1.3}), together with
a main theorem of \cite{munozozawa}, the fiber of the bundle in Definition \ref{def:9} must be planar and this gives a representation graph being planar.

We show the additional facts. (\ref{thm:3.1}) can be shown similarly and we omit a rigorous proof.

We show (\ref{thm:3.2}). For a suitable topologically quasi-trivial round fold map $f:M \rightarrow {\mathbb{R}}^2$, we can obtain a polyhedron PL homeomorphic to $W_f$ in the following way.
\begin{itemize}
\item First take $S^1 \times K_g$ where $K_g$ is a finite and connected graph we can embed in a closed, connected and orientable surface of genus $g$ in the PL category and as a result compact and connected surface ${\Sigma}_{g,1}$ obtained by removing the interior of a copy of the $2$-dimensional unit disk $D^2$ smoothly embedded in the surface.
\item Attach the disjoint union of finitely many copies of the $2$-dimensional unit disk $D^2$ by a PL homeomorphism from the boundary onto $S^1 \times V_g \subset S^1 \times K_g$ where $V_g$ is a finite subset. 
\end{itemize}
We have a polyhedron PL homeomorphic to $W_f$ and in the PL category, this can be embedded in a $3$-dimensional closed, connected and orientable manifold obtained by gluing the product manifold $S^1 \times {\Sigma}_{g,1}$ and $D^2 \times S^1$ on the boundaries in a suitable way. As presented before, in the case $g=0$, we can have a copy of the $3$-dimensional unit sphere as the $3$-dimensional manifold.

This completes the proof. 
\end{proof}

\begin{Rem}
	\label{rem:1}
	\cite{matsuzakiozawa} implies that a $2$-dimensional normal polyhedron, defined in Definition \ref{def:5}, can be embedded in some $3$-dimensional closed, connected and orientable manifold in the PL category. According to this, a $2$-dimensional Reeb space may not be embedded in any $3$-dimensional closed, connected and orientable manifold, whereas it is shown to be embedded in ${\mathbb{R}}^4$ (in the PL category). According to the discussion there, FIGURE 7 (c) of \cite{kobayashisaeki}, presented before, is for this.
\end{Rem}
\begin{Rem}
	\label{rem:2}
	In \cite{kitazawasaeki}, a {\it directed} round fold map is defined as a round fold map $f:M \rightarrow {\mathbb{R}}^2$ on a $3$-dimensional closed, connected and orientable manifold $M$ into ${\mathbb{R}}^2$ such that as we go to the origin $0 \in {\mathbb{R}}^2$ from the outermost connected component of $({\phi}_{{\mathbb{R}}^2} \circ f)(S(f))$, the numbers of connected components of the preimages increase where we abuse the notation of Definition \ref{def:2}. A graph manifold is shown to admit such a round fold map if and only if there exists a representation graph for the manifold which collapses to a point. Our Main Theorem \ref{mthm:2} and Theorem \ref{thm:3} characterize new wider classes of graph manifolds. 
	For the $3$-dimensional unit sphere, the total space of a smooth bundle over the $2$-dimensional unit sphere $S^2$ whose fiber is a circle and Lens spaces, representation graphs collapsing to points exist. As other examples, {\it Seifert manifolds over $S^2$}, which we do not define rigorously, enjoy the property.

    There exist graph manifolds for which no representation graphs collapsing to points exist and for which planar representation graphs exist according to \cite{kitazawasaeki,neumann}.

\end{Rem}

In addition to and related to Remarks \ref{rem:1} and \ref{rem:2}, we present concluding remarks.

According to \cite{neumann}, we have a complete invariant for graph manifolds whose values are finite graphs with some additional data such as integers. We call such a graph with additional data for a graph manifold $M$ a {\it normal form} of $M$. The graph is different from our representation graph. However, in some respects, they resemble. For example, we can know that if the graph in the normal form does not collapse to a point, then representation graphs for the manifold cannot collapse to points. Related to our Main Theorem \ref{mthm:2}, we can also know that if the graph in the normal form is not planar, then representation graphs for the manifold cannot be planar. We can have such manifolds by related theory. Related to Theorem \ref{thm:3}, for graphs we cannot embed in ${\mathbb{R}}^2$ or $S^2$ and we can embed in a closed, connected and orientable surface of genus $g>0$ in the PL category, we can argue similarly.

We close the paper by explaining about invariants via the embeddability of the Reeb spaces of simple fold maps into ${\mathbb{R}}^2$ in $3$-dimensional closed, connected and orientable manifolds.

\begin{itemize}
\item For the Reeb space of a simple fold map into ${\mathbb{R}}^2$ on a fixed graph manifold, we can define the minimal Heegaard genus for $3$-dimensional closed, connected and orientable manifolds where we can embed this (in the PL category). We have the minimal integer for such Reeb spaces. This value is always $0$ by virtue of Theorem \ref{thm:1}. This gives a trivial invariant for graph manifolds. 
\item For the Reeb space of a topologically quasi-trivial round fold map into ${\mathbb{R}}^2$ on a fixed graph manifold, we can define the minimal Heegaard genus for $3$-dimensional closed, connected and orientable manifolds where we can embed this (in the PL category). We have the minimal integer for such Reeb spaces. This value may not be $0$ by virtue of Theorem \ref{thm:3} (with \cite{neumann}). This gives a non-trivial invariant for graph manifolds. 
\end{itemize}

\section{Acknowledgment and data availability.}
The author is a member of JSPS KAKENHI Grant Number JP17H06128 "Innovative research of geometric topology and singularities of differentiable mappings" (Principal Investigator: Osamu Saeki) and the present work is supported by this. 

We declare that all data supporting our present study essentially are in the present paper. 

\end{document}